\documentclass[letterpaper,10pt]{amsart}
\usepackage{amssymb}
\usepackage[svgnames]{xcolor}
\usepackage{hyperref}
\usepackage{enumitem}
\usepackage{mathtools}
\usepackage{geometry}
\usepackage{amsfonts}
\usepackage{amsmath}
\usepackage{amsthm}
\usepackage{thmtools}
\usepackage{mathrsfs}
\usepackage{esint,subfig}
\usepackage{latexsym}
\usepackage[numbers]{natbib}
\usepackage{bm,bbm}
\usepackage{caption}
\usepackage{float}
\usepackage{multicol}
\usepackage[graphicx]{realboxes}

\definecolor{carminepink}{rgb}{0.92, 0.3, 0.26}

\hypersetup{colorlinks=true,allcolors=carminepink} 

\newtheorem{theorem}{Theorem}[section]
\newtheorem{lemma}[theorem]{Lemma}

\theoremstyle{definition}

\newtheorem{condition}[theorem]{Condition}

\theoremstyle{remark}
\newtheorem{remark}[theorem]{Remark}

\numberwithin{equation}{section}
\newtheorem*{acknow*}{Acknowledgments}

\newcommand*\diff{\mathop{}\!\mathrm{d}}
\definecolor{bittersweet}{rgb}{1.0, 0.44, 0.37}

\newcommand{\reals}{\mathbb{R}}					          
\newcommand{\integers}{\mathbb{Z}}
\newcommand{\naturals}{\mathbb{N}}
\newcommand{\Ex}{\mathbb{E}}

\newcommand{\sphere}{\mathbb{S}^2}				         
   
\newcommand{\alm}{a_{\ell, m}} 
\newcommand{\Y}{Y_{\ell, m}}

\newcommand{\argmax}{\text{argmax}}
\newcommand{\argmin}{\text{argmin}}
\newcommand{\primen}{\ell^{\prime}, m^{\prime}}

\newcommand{\Ltwo}{{L^2\bra{\sphere}}}

\definecolor{electricultramarine}{rgb}{0.25, 0.0, 1.0}

\providecommand{\abs}[1]{\left\vert#1\right\vert}			
\providecommand{\norm}[1]{\left\Vert#1\right\Vert}			

\providecommand{\bra}[1]{\left(#1\right)}

                                     \newcommand{\legf}{P_{\ell,m}}
                           
\newcommand{\summ}{\sum_{m=-\ell}^{\ell}}

 \newcommand{\Stwo}{\mathbb{S}^2}

\definecolor{rred}{RGB}{152,0,0}

\newcommand{\ARuno}{\text{SPHAR}\left(1\right)}
\newcommand{\legp}{P_{\ell}}
\calclayout
\newcommand{\Cl}{C_{\ell}}

\begin{document}
\bibliographystyle{alpha}
	
	\title{Parametric estimation for functional autoregressive processes on the sphere}
	
	\author{Alessia Caponera}
	\address{Institut de Mathématiques - Ecole Polytechnique F\'eed\'erale de Lausanne}
	\email{alessia.caponera@epfl.ch}
	
	\author{Claudio Durastanti}
	\address{Department S.B.A.I. - Sapienza University of Rome}
	\email{claudio.durastanti@uniroma.it}
	
	
	\date{19/JUL/2021}
	
	\dedicatory{}
	
	
	\begin{abstract}
	The aim of this paper is to define a nonlinear least squares estimator for the spectral parameters of a spherical autoregressive process of order 1 in a parametric setting. Furthermore, we investigate on its asymptotic properties, such as weak consistency and asymptotic normality.
	\end{abstract}
	
	\maketitle

\section{Introduction}\label{sec:intro}
In this paper, we propose a nonlinear least squares (NLS) estimator of the spectral parameters of a class of functional autoregressive processes of order 1, defined on the space of real-valued square-integrable functions over the unit sphere $\Ltwo$, see for example \cite{bosq}.\\ 
The spherical autoregressive model of order $1$ (from now on $\ARuno$) have been discussed by \cite{cm19} (see also \cite{splasso,spharma}), and comply with the output field $T\left(\cdot,t\right)$ described as an infinite-dimensional linear transformation of its previous realization summed to an independent spherical white noise $Z\left(\cdot,t\right)$, see \cite[Definition 3]{cm19} and also \cite{yadrenko}. More rigorously, the model is defined by 
\begin{equation}\label{eq:campo}
	T\left(x,t\right) = \Phi T\left(\cdot, t-1 \right)
	\left(x\right) + Z\left(x,t\right ),\quad \left(x,t\right) \in \Stwo\times \integers,
\end{equation}
where the autoregressive kernel operator $\Phi :\Ltwo  \to \Ltwo$ is given by
\begin{equation}\label{eq:kernelop} \left(\Phi f \right) \left(x\right) = \int _{\Stwo} k\left(\langle x,y \rangle \right) f\left(y\right) \diff y, \quad f\in \Ltwo,
\end{equation}
for some continuous $k:\left[-1,1\right]  \to \reals$, the so-called autoregressive kernel. Note that $k$ is said to be isotropic, since it depends only on the standard inner product on $\reals^3$, $\langle \cdot,\cdot\rangle$. As a direct consequence, the following spectral representation holds 
(in the $L^2$-sense)
\begin{equation}
	k\left(\langle x,y\rangle\right)=\sum_{\ell \in \naturals} \phi_\ell \frac{2\ell+1}{4\pi} P_\ell \left(\langle x,y\rangle\right),
\end{equation}
where $P_\ell:\left[-1,1\right] \to \reals$ denotes the Legendre polynomial of order $\ell$, while $\{ \phi_\ell:\ell \in \naturals \}$ is the set of the eigenvalues of the operator $\Phi$. \\
In particular, this work is concerned with spatially isotropic and temporally stationary sphere-cross-time random fields. In this case, $\left\{T\left(x,t\right): \left(x,t\right) \in \Stwo\times \integers \right\}$ can be read as the linear combination of spherical harmonics $\{\Y: \ell \in \naturals, m=-\ell,\ldots,\ell\}$, weighted by the corresponding time-varying harmonic coefficients, i.e.,
\begin{equation}\label{eq:harmexp}
	T\left(x,t\right) = \sum_{\ell \in \naturals} \sum_{m=-\ell}^{\ell} \alm \left(t\right)\Y \left(x\right), \quad\left(x,t\right) \in \Stwo\times \integers, 
\end{equation}
where, for fixed $t \in \integers$, $\{\alm\left(t\right): \ell \in \naturals, m=-\ell,\ldots,\ell\}$ are uncorrelated random variables, given by the standard inner product over the sphere
 \begin{equation}\label{eq:harmcoeff}
 	\alm\left(t\right)=\langle T\left(\cdot, t\right),\Y\rangle_{\Ltwo}.
 \end{equation}
Moreover, if $T$ is a solution of the autoregressive equation $\eqref{eq:campo}$, the harmonic coefficients satisfy 
\begin{equation}\label{eq:ARuno}
	\alm\left(t\right) = \phi_\ell \alm\left(t-1\right) + a_{\ell,m;Z}\left(t\right),
\end{equation}
with $\left\{a_{\ell,m;Z}\left(t\right):\ell \in \naturals, m=-\ell,\ldots, \ell\right\}$ being the harmonic coefficients of the spherical white noise $Z\left(\cdot,t\right)$, at the time $t$.\\
In our setting, we choose a regularly varying condition, that is, a parametric model on the structure of the autoregressive kernel,
\begin{equation}\label{eq:condit}
	\phi_\ell = G \ell^{-\alpha}, \quad G\in (-1,1)\backslash \{0\},\, \alpha \in \left(1,\infty\right), \,\ell \in \naturals,
\end{equation}
see Condition \ref{cond:semiparam} in Section \ref{sec:prel}.\\
An analogous parametric condition is assumed to hold also for the eigenvalues of the covariance operator of the spherical random field $Z$ (see also \cite{dlm}).

 A method commonly used in parametric settings for finite-dimensional parameters of interest is the nonlinear least squares (NLS). NLS belongs to the class of estremum estimators, obtained as the result of a maximization procedure of a given objective function depending on data and sample size (see, for example, \cite{NmcF}). This class includes, among others, maximum likelihood, Whittle, generalized method of moments and minimum distance estimators. Under mild assumptions, they are characterized by some relevant properties, such as weak consistency and asymptotic Gaussianity (see \cite{am,hayashi}). Some of these methods have already been successfully applied to the study of purely spatial spherical random fields. For example, in \cite{dlm,dlmejs} the authors provide Whittle-like estimators for the spectral index of a Gaussian and isotropic random field over the sphere, both in the harmonic and in the wavelet domains. Furthermore, their weak consistency and asymptotic convergence to Gaussianity have been proved in the high frequency limit. 
 
 Fixed the truncation multipole $L>0$, we consider the truncated random field $T_L$ as the sum of the first $L$ components of $T$, so that \eqref{eq:harmexp} becomes
\begin{equation}\label{eq:truncfield}
T_L\left(x,t\right) = \sum_{\ell = 1}^{L} \sum_{m=-\ell}^{\ell} \alm \left(t\right)\Y \left(x\right), \quad\left(x,t\right) \in \Stwo\times \integers.
\end{equation} 
Here, merging \eqref{eq:ARuno} and \eqref{eq:truncfield} yields the following objective function 
\begin{equation}\label{eq:obj}
		R_{L,N} (G , \alpha )= \frac{1}{N} \sum_{t=1}^N \norm{T_L\left(\cdot,t\right)-\Phi T_L\left(\cdot,t-1\right)}^2_{\Ltwo}  =\frac{1}{N} \sum_{t=1}^N \sum_{\ell=1}^L \sum_{m=-\ell}^\ell \left ( \alm(t)  - \phi_{\ell}\alm (t-1) \right )^2.	
\end{equation}
Since we assume that $L=L_N$, from now on the objective function is labeled by $R_{N} (G , \alpha  )$, omitting the dependence on $L$. Imposing the Condition \eqref{eq:condit}, the estimator for the parameter $\left(G,\alpha\right)$ is given by
\begin{equation}\label{eq:estimazzi}
	( \widehat{G}_{N}, \widehat{\alpha}_{N}) = \underset{\left(G,\alpha\right)\in \Theta}{\argmin} \, R_{N} (G , \alpha ),
\end{equation}
where $N >1$ is the highest time at which $T$ is observed.\\
We will show that, under mild conditions, the estimator \eqref{eq:estimazzi} is consistent and asymptotically Gaussian. We remark that the asymptotic framework here considered is quite different from usual and can be related to the one proposed by \cite{cm19,splasso} (see also \cite{dlm}). We assume to collect sequentially over time spherical observations which are a realization of an isotropic and stationary field $T$. In this sense, the asymptotics here considered is respect to higher and higher resolution data becoming available as $N$ grows to infinity.  

The plan of the paper is as follows: Section \ref{sec:prel} recalls some useful results concerning harmonic analysis and spherical autoregressive processes. The main contributions of this paper are collected in Section \ref{sec:main}, while Section \ref{sec:salmazzi} contains the proofs of some auxiliary results. 

\section{Preliminaries}\label{sec:prel}
This section collects some results concerning harmonic analysis on the sphere and the construction of space-time spherical random fields. Further details on harmonic analysis on the sphere and sphere-cross-time random fields can be found, among others, in \cite{steinweiss,vilenkin,MP:11} and \cite{gneiting, jun, porcu, steinst} respectively.\\
From now on, we will make use of the following notation. For a set of random variables $\left\{X_n\right\}_{n\in \naturals}$, the notation $X_{n}=o_{p}(1)$ denotes that for any $\epsilon>0$, it holds that $
\lim_{n\rightarrow \infty} \Pr\left(\abs{X_n}>\epsilon\right)=0.$ Let $\left\{c_n\right\}_{n\in \naturals}$ be a real-valued, deterministic sequence; then, with 
	$X_{n}=o_{p}(c_{n})$ we will indicate that $X_n/c_n = o_p\left(1\right)$. 
	The notation $X_{n}=O_{p}(c_{n})$
	means that $\left\{X_n/c_n\right\}_{n \in \naturals}$ is stochastically bounded, that is, for any $\epsilon>0$, there exist $0<M,N<\infty$ such that, for all $n>N$, it holds that
	$\Pr\left(\abs{X_{n}/c_{n}}>M\right)<\epsilon.$
	
	Let us denote a point belonging to the unit sphere with $x \in \Stwo $. It can be also identified by two angular coordinates, so that $x=\left(\vartheta,\varphi\right)$, where $\vartheta \in \left[0,\pi\right]$ and $\varphi \in  \left[ \left.0,2\pi \right)\right.$ are the colatitude and longitude respectively. The spherical Lebesgue measure is labeled by $\diff x=\sin\vartheta \diff\vartheta \diff\varphi$, while $\Ltwo=L^2\left(\Stwo,\diff x \right)$ is the space of real-valued square-integrable functions over the sphere with respect to $\diff x$. A standard orthonormal basis for $\Ltwo$ is given by the set of spherical harmonics $\{\Y : \ell\in \naturals; m=-\ell,\ldots,\ell \}$ (see, for example, \cite{MP:11, steinweiss, vilenkin}). We refer to the index $\ell \in \naturals$ as the multipole, while $m=-\ell, \ldots,\ell$ is the ``azimuth'' number. In this paper we will make use of the so-called real spherical harmonics. More specifically, for any $\ell \in \naturals$ and $m =-\ell,\ldots, \ell$, the spherical harmonic $\Y: \mathbb{S}^2 \to \reals$ is given by the normalized product of the so-called Legendre associated function $\legf: \left[-1,1\right]\rightarrow \reals$ of degree $\ell$ and order $m$, depending only on $\vartheta$ and given by
\begin{equation*}
	\legf \left(u\right)=\frac{1}{2^\ell \ell !} \left(1-u^2\right)^{\frac{m}{2}}\frac{\diff^{\ell + m }}{\diff u^{\ell+m}} \left(u^2-1\right)^{\ell}, \quad u \in \left[-1,1\right],
\end{equation*} 
and a trigonometric function depending only on $\varphi$, i.e.,
\begin{align*}
	\Y \left(\vartheta, \varphi\right)= \begin{cases}
		\sqrt{\frac{\left(2\ell +1\right)}{2 \pi} \frac{\left(\ell-m\right)!}{\left(\ell+m\right)!}} \legf \left(\cos \vartheta \right)\cos\left(m \varphi\right) &\text{for } m \in \{1,\ldots, \ell \}\\
		\sqrt{\frac{\left(2\ell +1\right)}{4 \pi}} \legp  \left(\cos \vartheta \right) &\text{for } m = 0 \\
		\sqrt{\frac{ \left(2\ell +1\right)}{2 \pi} \frac{\left(\ell+m\right)!}{\left(\ell-m\right)!}} P_{\ell,-m} \left(\cos \vartheta \right)\sin\left(-m\varphi\right) &\text{for } m \in \{-\ell,\ldots, -1 \}
	\end{cases}.
\end{align*} 
Spherical harmonics display also the following addition formula
\begin{equation*}\label{eq:addition}
	\summ \Y\left(x\right)\Y\left(y\right)=\frac{2\ell+1}{4\pi}\legp \left(\langle x,y\rangle\right), \quad x,y\in \Stwo,
\end{equation*}
where $\legp:[-1,1] \to \reals$ is the Legendre polynomial of order $\ell$, given by
$$
\legp\left(u\right)=\frac{1}{2^{\ell}\ell !} \frac{\diff ^ \ell}{\diff u ^ \ell} \left(u^2-1\right)^\ell, \quad u \in \left[-1,1\right].
$$

Given a probability space $\left(\Omega, \mathcal{F},\Pr\right)$, we consider a sphere-cross-time random field, that is, a real-valued collection of random variables $$\{T\left(x,t\right):\left(x,t\right)\in \Stwo \times \integers \}.$$
In this paper, $T$ is assumed to be real-valued, centered, mean-square continuous, and Gaussian. Additionally, the random field is isotropic in the spatial domain and stationary in the time domain. We recall that a spherical random field is isotropic when invariant in distribution with respect to rotations, and stationary if its stochastic properties do not vary over time. More in detail, it holds that 
\begin{equation}\label{eq:isostat}
	T\left(R\,\cdot, \cdot +\tau\right) \overset{d}{=} 	T\left(\cdot, \cdot\right), 
\end{equation}
with $\tau \in \integers$, and $R \in SO\left(3\right)$, the special group of rotations, and the symbol $\overset{d}{=}$ denotes equality in distribution. 
Under these assumptions, as described in Equation \eqref{eq:harmexp}, we have the following spectral representation
\begin{equation*}
	T\left(x,t\right) = \sum_{\ell \in \naturals}\sum _{m=-\ell}^{\ell} \alm\left(t\right)\Y \left(x\right), \quad \left(x,t\right)\in \Stwo \times \integers\,,
\end{equation*}
where, for any $t \in \integers$, the set of the harmonic coefficients $\{\alm\left(t\right):\ell \in \naturals, m=-\ell ,\dots,\ell\}$, given by \eqref{eq:harmcoeff}, can be explicitly computed by
\begin{equation*}	
	\alm\left(t\right)=\int_{\Stwo}T\left( x,t\right) \Y \left( x\right) \diff x\,,
\end{equation*}%
and collects all the stochastic information of $T$ at the time $t$.
Since $\Ex\left[T\left(x,t\right)\right]=0$ for all $\left(x,t\right) \in \Stwo \times \integers$, it holds that 
$$
\Ex\left[\alm\left(t\right)\right]= 0 \quad \text{for }\ell \in \naturals, \quad m = -\ell,\ldots, \ell, \quad t\in \integers\,.
$$
Furthermore, let $\Gamma:\left(\Stwo\times \mathbb{Z}\right)\times \left(\Stwo\times \mathbb{Z}\right) \rightarrow \mathbb{R}$ be the covariance function of $T$. Since the space-time spherical random field is isotropic and stationary, then there exists a function $\Gamma_0: \left[-1,1\right]\times \mathbb{Z} \rightarrow \mathbb{R}$, so that
\begin{equation*}
	\Gamma\left(x,t,y,s\right) = \Gamma_0\left(\langle x,y\rangle, t-s \right), \quad \left(x,t\right),\left(y,s\right) \in \Stwo\times \mathbb{Z}\,.
\end{equation*}
Also, the dependence structure of $T$ is fully characterized by the one of its harmonic coefficients, that is, 
\begin{equation}\label{eq:cltime}
	\Ex\left[\alm\left(t\right) {a}_{\primen}\left(s\right)\right] = \Cl\left(t-s\right)\delta_{\ell}^{\ell^\prime}\delta_{m}^{m^\prime}, \quad t,s \in \integers\,,
\end{equation}
for any $\ell, \ell^\prime \in \naturals$, $m=-\ell,\ldots,\ell$, $m^\prime= -\ell^\prime, \ldots, \ell^\prime$.
This leads to the following spectral decomposition in terms of Legendre polynomials
\begin{equation}\label{eq:specdec}
	\Gamma\left(x,t,y,s\right) = \sum_{\ell \in \naturals} \Cl \left(t-s\right) \frac{2\ell+1}{4\pi} P_\ell \left(\langle x,y \rangle \right), \quad \left(x,t\right),\left(y,s\right)\in \Stwo \times \integers\,,
\end{equation}
see also \cite{bergporcu,spharma}.\\
Note that, if $t=s$, $\{\Cl\left(0\right): \ell \in \naturals\}$ in \eqref{eq:cltime} correspond to the eigenvalues of the covariance operator of a purely spatial spherical random field, the so-called angular power spectrum, (see, for example, \cite[Remark 5.15, p.124; Remark 6.16, p.147]{MP:11}).

If $T$ is $\ARuno$, with $\abs{\phi_\ell}<1$ for each $\ell\in \naturals$, the following formula for the spectrum of $\Gamma$ holds
\begin{equation}
\label{eq:cielle} \Cl\left(t-s\right)=\frac{C_{\ell;Z}}{1-\phi_\ell^2}\phi_\ell^{\abs{t-s}},
\end{equation}
where $C_{\ell;Z}=\Ex\abs{a_{\ell,m;Z}\left(t\right)}^2$ is the angular power spectrum of $Z$.\\
In our setting, we focus on a parametric model on the set $\{ \phi_\ell:\ell \in \naturals \}$, described by the following condition.
\begin{condition}\label{cond:semiparam} Consider an isotropic and stationary $\ARuno$ process as given by \eqref{eq:campo} and \eqref{eq:kernelop}.
	The eigenvalues of the autoregressive operator $\Phi$ $\{ \phi_\ell:\ell \in \naturals \}$ are such that 
	\begin{equation*}
		\phi_\ell = G \ell^{-\alpha}, \quad  \ell \in \naturals,
	\end{equation*}
	where $1< a_1\le \alpha \le a_2$, with $1<a_1<a_2<\infty$, and 
	$
	0 < |G| <1.
	$\\
Moreover, the angular power spectrum of the spherical white noise $Z$ $\left\{C_{\ell;Z}: \ell \in \naturals\right\}$ satisfies
	\begin{equation*}
	C_{\ell;Z} = H \ell^{-\gamma} , \quad  \ell \in \naturals,
\end{equation*}
	where $ \gamma > 2$ and $H >0$.
\end{condition}
As far as the parameter $\alpha$ is concerned, choosing $1<a_1<a_2<\infty$ ensures the square-summability of the $\phi_\ell$'s, i.e., the operator $\Phi$ is Hilbert-Schmidt with $$\sum_{\ell \in \naturals} (2\ell+1) |\phi_{\ell}|^2 < \infty,$$ and the consistency of the estimator here presented, as discussed below in Section \ref{sec:main}.\\
About the parameter $G$, observe that the cases $G\in (0,1)$ and $G\in (-1,0)$ correspond to a positive or negative definite operator $\Phi$ respectively.
\begin{remark} 
Our construction resembles in the space-time setting the so-called Legendre-Matérn covariance function, defined in a purely spatial framework in \cite{GUINNESS2016}. Indeed, $\gamma$ and $\alpha$ are smoothness parameters while the constants $H$ and $G$ control the scale of \eqref{eq:specdec}. 
\end{remark}
\section{Least squares estimates in the parametric setting}\label{sec:main}
In this section we will discuss the construction of the nonlinear least squares estimator for the spectral parameter $\theta$ of the eigenvalues $\phi_\ell=\phi_\ell\left(\theta\right)$ of the autoregressive kernel $k$ in a parametric setting. Here the spectral parameter $\theta = \left(G,\alpha\right)$ is defined over the parameter space $\Theta = \left(-1,1\right) \backslash \{0\} \times \left[a_1,a_2\right]$, $1<a_1<a_2<\infty$, and following Condition \ref{cond:semiparam} yields 
$$
\phi_\ell = G \ell^{-\alpha}, \quad \ell \in \naturals.
$$
The true values for the parameter $\theta$ to be estimated are labeled by $(G_0, \alpha_0)$.
The estimation procedure can be formalized as follows:
\begin{equation*}
( \widehat{G}_N, \widehat{\alpha}_N) = \underset{\left(G,\alpha\right)\in \Theta}{\argmin} \, R_{N} (G , \alpha ),
\end{equation*}
where $R_{N}$ is the objective function given by Equation \eqref{eq:estimazzi}, 
\begin{equation*}
 R_{N} (G , \alpha  )=  \frac{1}{N} \sum_{t=1}^N \sum_{\ell=1}^{L_N} \sum_{m=-\ell}^\ell \left ( \alm(t)  - G \ell^{-\alpha }\alm (t-1) \right )^2.
\end{equation*}
\begin{condition}\label{cond:tronc}
The truncation multiple $L_N$ is chosen such that $L_N \to \infty$ and $(\log L_N)^2 / \sqrt{N} \to 0,$ as $N\to \infty$.
\end{condition}
Throughout this paper we will make extensive use of the following quantities
\begin{align*}
&\widehat{U}_{N} (\alpha)=  \frac{1}{N} \sum_{t=1}^{N} \sum_{\ell=1}^{L_N} \sum_{m=-\ell}^{\ell}  a_{\ell,m}\left(t\right)a_{\ell,m}\left(t-1\right) \ell^{-\alpha} ; && \widehat{D}_{N}(\alpha) = \frac{1}{N} \sum_{t=1}^{N} \sum_{\ell=1}^{L_N} \sum_{m=-\ell}^{\ell} \left \vert a_{\ell,m}\left(t-1\right) \right \vert ^{2} \ell^{-2 \alpha};\\
	&\widehat{U}_{N}^\prime (\alpha)= - \frac{1}{N} \sum_{t=1}^{N} \sum_{\ell=1}^{L_N} \sum_{m=-\ell}^{\ell}  a_{\ell,m}\left(t\right)a_{\ell,m}\left(t-1\right) \ell^{-\alpha} \log \ell; && \widehat{D}^\prime_{N}(\alpha) = - \frac{2}{N} \sum_{t=1}^{N} \sum_{\ell=1}^{L_N} \sum_{m=-\ell}^{\ell} \left \vert a_{\ell,m}\left(t-1\right) \right \vert ^{2} \ell^{-2 \alpha} \log \ell;\\
	& \widehat{U}_{N}^{\prime\prime}(\alpha)=  \frac{1}{N} \sum_{t=1}^{N} \sum_{\ell=1}^{L_N} \sum_{m=-\ell}^{\ell}  a_{\ell,m}\left(t\right)a_{\ell,m}\left(t-1\right) \ell^{-\alpha} \log^2 \ell; && \widehat{D}_{N}^{\prime \prime}(\alpha) = \frac{4}{N} \sum_{t=1}^{N} \sum_{\ell=1}^{L_N} \sum_{m=-\ell}^{\ell} \left \vert a_{\ell,m}\left(t-1\right) \right \vert ^{2} \ell^{-2 \alpha} \log^2 \ell,
\end{align*}
and their corresponding expected values
\begin{align*}
	&U_{N} (\alpha)=  \sum_{\ell=1}^{L_N} \left(2\ell+1\right) C_\ell\left(1\right) \ell^{-\alpha}; && D_{N}(\alpha) =  \sum_{\ell=1}^{L_N} \left(2\ell+1\right) C_\ell\left(0\right) \ell^{-2 \alpha};\\
	&U_{N}^\prime (\alpha)= - \sum_{\ell=1}^{L_N} \left(2\ell+1\right) C_\ell\left(1\right)  \ell^{-\alpha} \log \ell; && D^\prime_{N}(\alpha) =-2 \sum_{\ell=1}^{L_N} \left(2\ell+1\right) C_\ell\left(0\right) \ell^{-2 \alpha} \log \ell;\\
	& U_{N}^{\prime\prime}(\alpha)=  \sum_{\ell=1}^{L_N} \left(2\ell+1\right) C_\ell\left(1\right)  \ell^{-\alpha} \log^2 \ell; && D_{N}^{\prime \prime}(\alpha) =4 \sum_{\ell=1}^{L_N}\left(2\ell+1\right) C_\ell  \left(0\right)\ell^{-2 \alpha} \log^2 \ell,
\end{align*}
where
\begin{align*}
	C_\ell\left(1\right)=C_\ell\left(0\right)\phi_\ell = G_0 \eta_\ell \ell^{-\gamma-\alpha_0}; \quad C_\ell\left(0\right)= \frac{C_{\ell;Z}}{1-\phi_\ell^2}=\eta_\ell \ell^{-\gamma},
\end{align*}
with $0 < \upsilon_1 \le \eta_\ell \le \upsilon_2<\infty$. For $\alpha >1 $, these last deterministic sums are absolutely and uniformly convergent. Indeed, for the greater terms $U_{N}^{\prime\prime}$ and $D_{N}^{\prime\prime}$, we have that 
\begin{equation*}
	\sum_{\ell=1}^{L_N} | \left(2\ell+1\right) C_\ell\left(1\right)  \ell^{-\alpha} \log^2 \ell | \le \sum_{\ell=1}^\infty | \left(2\ell+1\right) C_\ell\left(1\right)  \ell^{-1} \log^2 \ell | < \infty
\end{equation*}
and
\begin{equation*}
	\sum_{\ell=1}^{L_N} | \left(2\ell+1\right) C_\ell\left(0\right)  \ell^{-2\alpha} \log^2 \ell | \le \sum_{\ell=1}^\infty | \left(2\ell+1\right) C_\ell\left(0\right)  \ell^{-2} \log^2 \ell | < \infty.
\end{equation*}
Their limits will be denoted by $U(\alpha), U'(\alpha), U''(\alpha), D(\alpha), D'(\alpha), D''(\alpha)$.

Now, observe that solving the following equation
\begin{equation*}
\frac{\partial R_{N}}{\partial G} = -\frac{2}{N} \sum_{t=1}^{N} \sum_{\ell=1 }^{L_N}\sum_{m=-\ell}^{\ell} \left (\alm(t) - G \ell^{-\alpha}\alm (t-1) \right ) \left  ( \ell^{-\alpha} \alm(t-1) \right ) = 0,
\end{equation*}
with respect to $G$, leads to
\begin{equation*}
G^*  = \frac{\sum_{t=1}^{N} \sum_{\ell=1 }^{L_N}\sum_{m=-\ell}^{\ell} \alm(t) \alm(t-1)  \ell^{-\alpha}}{ \sum_{t=1}^{N} \sum_{\ell=1 }^{L_N}\sum_{m=-\ell}^{\ell}| \alm(t-1)|^2 \ell^{-2\alpha}} = \frac{\sum_{\ell=1}^{L_N} \ell^{-\alpha} ( 2\ell+1) \widehat{C}_{\ell;N}(1)}{ \sum_{\ell=1}^{L_N} \ell^{-2\alpha} (2\ell+1) \widehat{C}_{\ell;N}(0)},
\end{equation*}
where
\begin{equation*}
\widehat{C}_{\ell;N}(\tau) = \frac{1}{N(2 \ell+1)} \sum_{t=1}^N \sum_{m=-\ell}^\ell  \alm(t-1+\tau)\alm(t-1), \qquad  \widehat{\phi}_{\ell;N} = \frac{\widehat{C}_{\ell;N}(1)}{\widehat{C}_{\ell;N}(0)}.
\end{equation*}
Thus, we have
\begin{align}\label{eq:gstar}
R_{N} (G^*, \alpha ) &= \frac{1}{N} \sum_{t=1}^N \sum_{\ell=1}^{L_N} \sum_{m=-\ell}^\ell \left ( \alm(t)  - G^* \ell^{-\alpha}\alm (t-1) \right )^2 \nonumber\\
&= \frac{1}{N} \sum_{t=1}^N \sum_{\ell=1}^{L_N} \sum_{m=-\ell}^\ell\alm^2(t) -   \frac{\widehat{U}_{N}^2(\alpha)}{\widehat{D}_{N}(\alpha)} .
\end{align}
The minimization problem can be equivalently written as
\begin{align*}
 \underset{\alpha \in A}{\argmin}  \left \{ \frac{1}{N}  \sum_{t=1}^N \sum_{\ell=1}^{L_N} \sum_{m=-\ell}^\ell  \alm^2(t) -  \frac{\widehat{U}_{N}^2(\alpha)}{\widehat{D}_{N}(\alpha)} \right \}& =  \underset{\alpha \in A}{\argmax}  \, \frac{\widehat{U}_{N}^2(\alpha)}{\widehat{D}_{N}(\alpha)} = \underset{\alpha \in A }{\argmax}  \left \{ 2 \log \widehat{U}_{N}(\alpha) - \log \widehat{D}_{N}(\alpha) \right\},
\end{align*}
where $A=[a_1, a_2]$.
We will call 
\begin{equation}\label{eq:objred}
\widetilde{R}_{N} (\alpha) =  2 \log \widehat{U}_{N}(\alpha) - \log \widehat{D}_{N}(\alpha)\end{equation} the \emph{reduced objective function}.


\subsection{Consistency}\label{sec:cons}
In this section we will prove the weak consistency for the estimators $\widehat{\alpha}_N$ and $\widehat{G}_N$, as stated in the following theorem.	
\begin{theorem}\label{th:cons}
	Under Conditions \ref{cond:semiparam} and \ref{cond:tronc}, as $N \rightarrow \infty$, it holds that
	\begin{align}
		\label{eq:alpha}\widehat{\alpha}_N&\overset{p}{\to}\alpha_0,\\	\label{eq:G}	\widehat{G}_N &\overset{p}{\to}G_0.
	\end{align}
\end{theorem}
To prove Theorem \ref{th:cons}, we will resort to a method developed by \cite{brillinger} and \cite{Robinson}. This technique has been already used to establish the weak consistency for the spectral parameters of a spherical random field in a purely spatial setting by \cite{dlm}. More in detail, let $\Delta \widetilde R_{N} (\alpha, \alpha_0)$ denote the difference between the reduced objective function $ \widetilde R_{N}$ given by \eqref{eq:objred} evaluated at the generic $\alpha\in\left[a_1,a_2\right]$ and at the true $\alpha_0$. 
Straightforward calculations lead to
\begin{align*}
	\Delta \widetilde R_{N} (\alpha, \alpha_0)  &=   \widetilde R_{N} (\alpha) -  \widetilde R_{N} (\alpha_0)  \\
	&= 2 \log \widehat{U}_{N}(\alpha)  -  2 \log \widehat{U}_{N}\left(\alpha_0\right)  - \left [ \log \widehat{D}_{N}(\alpha) -  \log \widehat{D}_{N}\left(\alpha_0\right) \right] \\
	& = 2 \log \frac{\widehat{U}_{N}(\alpha)}{U_{N}(\alpha)}  -  2 \log \frac{\widehat{U}_{N}\left(\alpha_0\right)}{U_{N}\left(\alpha_0\right)}  - \left [ \log \frac{\widehat{D}_{N}(\alpha)}{D_{N}(\alpha)} -  \log \frac{\widehat{D}_{N}\left(\alpha_0\right)}{D_{N}\left(\alpha_0\right)} \right] \\
	& - \left [ \log \frac{D_{N}(\alpha)}{D_{N}\left(\alpha_0\right)}   -  2 \log \frac{U_{N}(\alpha)}{U_{N}\left(\alpha_0\right)}  \right ] .
\end{align*}
Let us now define
\begin{align}
	\label{eq:T}	&T_{N}(\alpha, \alpha_0)  = 2 \log \frac{\widehat{U}_{N}(\alpha)}{U_{N}(\alpha)}  -  2 \log \frac{\widehat{U}_{N}\left(\alpha_0\right)}{U_{N}\left(\alpha_0\right)}  - \left [ \log \frac{\widehat{D}_{N}(\alpha)}{D_{N}(\alpha)} -  \log \frac{\widehat{D}_{N}\left(\alpha_0\right)}{D_{N}\left(\alpha_0\right)} \right] ,\\
	\label{eq:V}	&V_{N}(\alpha, \alpha_0)  =\log \frac{D_{N}(\alpha)}{D_{N}\left(\alpha_0\right)}   -  2 \log \frac{U_{N}(\alpha)}{U_{N}\left(\alpha_0\right)},  
\end{align}
so that
\begin{equation}\label{eq:consdec}
	\Delta \widetilde R_{N} (\alpha, \alpha_0) = T_{N}(\alpha, \alpha_0) - V_{N}(\alpha, \alpha_0). 
\end{equation}
In order to establish the consistency results presented in Theorem \ref{th:cons}, we will make use of the following auxiliary results, whose proofs are available below in Section \ref{sec:salmazzi}. We remark that these results hold under the same assumptions of Theorem \ref{th:cons}, stated in Conditions \ref{cond:semiparam} and \ref{cond:tronc}.
\begin{lemma}\label{lemma:V}
	For $\epsilon > 0$, let $B_\epsilon = \{\alpha: |\alpha- \alpha_0| < \epsilon \}.$ Let also $V_{N}\left(\alpha,\alpha_0\right)$ be given by \eqref{eq:V}. Then, there exists $V_\epsilon >0$ such that 	
	\begin{equation}\label{eq:liminf}
		\lim_{N\to \infty} \inf_{\overline{B}_\epsilon \cap A} V_{N}(\alpha, \alpha_0) = V_\epsilon.
	\end{equation}
\end{lemma}
\begin{lemma}\label{lemma:T}
For $N \to \infty,$ it holds that
	\begin{equation}\label{eq:limsup}
		\Ex \left [ \sup_A \left | \frac{\widehat{U}_{N} (\alpha) - U_{N}(\alpha)}{U_{N}(\alpha)} \right| \right]= O\left ( \frac{1}{\sqrt{N}}\right)  \qquad \text{and} \qquad  \Ex \left [\sup_A \left | \frac{\widehat{D}_{N} (\alpha) - D_{N}(\alpha)}{D_{N}(\alpha)} \right | \right]= O\left ( \frac{1}{\sqrt{N}}\right) .
	\end{equation}
\end{lemma}

\begin{lemma}\label{Ghat0} For $N\to \infty$,
	$$
	\frac{\widehat{U}_N(\alpha_0)}{\widehat{D}_N(\alpha_0)} \overset{p}{\to} G_0.
	$$
\end{lemma}

\begin{lemma}\label{lemma:loglog}
	For $N \to \infty$,
	$$
\log L_N |\widehat{\alpha}_N - \alpha_0| \overset{p}{\to} 0.
	$$
\end{lemma}

\begin{proof}[Proof of Theorem \ref{th:cons}]
Let us first prove \eqref{eq:alpha}. For $\epsilon > 0$, let $B_\epsilon = \{\alpha: |\alpha- \alpha_0| < \epsilon \}.$ Without loss of generality, choose $\epsilon >0$ such that $\overline{B}_\epsilon \cap A$ is non-empty. We have that
\begin{align}
	\notag	\Pr\left (|\widehat{\alpha}_N -\alpha_0| \ge \epsilon \right) &= \Pr \left ( \widehat{\alpha}_N\in \overline{B}_\epsilon \cap A \right) \le \Pr  \left ( \sup_{\overline{B}_\epsilon \cap A} \widetilde{R}_{N} (\alpha) \ge  \widetilde{R}_{N} (\alpha_0) \right) \\
	\notag	&=   \Pr  \left ( \sup_{\overline{B}_\epsilon \cap A} \Delta \widetilde R_{N} (\alpha, \alpha_0) \ge  0 \right)\\
	 &\le \Pr  \left( \sup_{ A}  |T_{N} (\alpha, \alpha_0)| \ge  \inf_{\overline{B}_\epsilon \cap A} V_{N}(\alpha, \alpha_0)   \right) .\label{eq:constup}
\end{align}
Using \eqref{eq:T}, for $M > 0$, we have that
\begin{align*}
 \Pr  \left( \sup_{ A}  |T_{N} (\alpha, \alpha_0)| \ge  4 M \right)  &\le \Pr  \left( \sup_{ A}  \left | \log \frac{\widehat{U}_{N}(\alpha)}{U_{N}(\alpha)}\right  | \ge   M/2 \right)  + \Pr  \left( \left   | \log \frac{\widehat{U}_{N}(\alpha_0)}{U_{N}(\alpha_0)}\right | \ge   M/2 \right)  \\&+\Pr  \left( \sup_{ A} \left  | \log \frac{\widehat{D}_{N}(\alpha)}{D_{N}(\alpha)} \right| \ge   M \right)  + \Pr  \left(  \left | \log \frac{\widehat{D}_{N}(\alpha_0)}{D_{N}(\alpha_0)} \right| \ge   M \right)  \\
 &\le 2 \Pr  \left( \sup_{ A}  \left | \log \frac{\widehat{U}_{N}(\alpha)}{U_{N}(\alpha)}\right  | \ge   M/2 \right)   + 2 \Pr  \left( \sup_{ A} \left  | \log \frac{\widehat{D}_{N}(\alpha)}{D_{N}(\alpha)} \right| \ge   M \right) .
\end{align*}
From the inequality $$|\log(1+x)| \le 2 |x|, \quad \text{for }|x| \le 1/2, $$ we deduce that, for any nonnegative random variable $Y$,
	$$
	\Pr ( |\log Y | \ge \delta) \le 2 \Pr (|Y-1| \ge 2 \delta) \qquad \text{for } \delta > 0.
	$$
As a consequence, using Markov inequality yields
\begin{align}
\notag \Pr  \left( \sup_{ A}  |T_{N} (\alpha, \alpha_0)| \ge  4 M \right)  & \le  4 \Pr  \left( \sup_{ A}  \left | \frac{\widehat{U}_{N} (\alpha) - U_{N}(\alpha)}{U_{N}(\alpha)} \right| \ge   M \right)   + 4 \Pr  \left( \sup_{ A}\left | \frac{\widehat{D}_{N} (\alpha) - D_{N}(\alpha)}{D_{N}(\alpha)} \right| \ge   2M \right)
 \\& \le \frac4M \, \Ex \left [ \sup_A \left | \frac{\widehat{U}_{N} (\alpha) - U_{N}(\alpha)}{U_{N}(\alpha)} \right| \right] + \frac2M \, \Ex \left [ \sup_A \left | \frac{\widehat{D}_{N} (\alpha) - D_{N}(\alpha)}{D_{N}(\alpha)} \right| \right]  \label{eq:trapizzino}.
 \end{align}
Thus, by Lemma \ref{lemma:V} and Lemma \ref{lemma:T}, we obtain the weak consistency for $\widehat\alpha_N$ \eqref{eq:alpha}.
\\ Let us now consider \eqref{eq:G}. By definition, we have
	$$
	\widehat{G}_N =  \frac{\widehat{U}_{N} (\widehat{\alpha}_N)}{\widehat{D}_{N} (\widehat{\alpha}_N) }  = \frac{\widehat{U}_{N} (\widehat{\alpha}_N)}{\widehat{D}_{N} (\widehat{\alpha}_N) }  \cdot \frac{D_N (\alpha_0)}{U_N (\alpha_0) } \cdot \frac{U_N (\alpha_0)}{D_N (\alpha_0) } = \frac{\widehat{U}_{N} (\widehat{\alpha}_N)}{U_N (\alpha_0)}  \cdot \frac{D_N (\alpha_0)}{\widehat{D}_N (\alpha_0) }  \cdot G_0.
	$$
	Hence, we need to show that, for $N \to \infty$,
	$$
	\frac{\widehat{U}_{N} (\widehat{\alpha}_N) - U_N (\alpha_0)}{U_N (\alpha_0)} \overset{p}{\to} 0,\qquad  \frac{\widehat{D}_{N} (\widehat{\alpha}_N) - D_N (\alpha_0)}{D_N (\alpha_0)} \overset{p}{\to} 0 .
	$$
	First of all, note that
	\begin{align*}
	\widehat{D}_{N} (\widehat{\alpha}_N)- D_N (\alpha_0)&=	\sum_\ell \ell^{-2\widehat{\alpha}_N} (2\ell+1) \widehat{C}_{\ell;N}(0)-\sum_\ell \ell^{-2{\alpha}_0} (2\ell+1){C}_{\ell}(0)\\ &=  \sum_\ell \ell^{-2(\widehat{\alpha}_N-\alpha_0)} \ell^{-2\alpha_0} (2\ell+1) \widehat{C}_{\ell;N}(0)-\sum_\ell \ell^{-2{\alpha}_0} (2\ell+1){C}_{\ell}(0)  \\
		&\pm  \sum_\ell \ell^{-2(\widehat{\alpha}_N-\alpha_0)} \ell^{-2\alpha_0} (2\ell+1) C_{\ell}(0)\\
		&= \sum_\ell \ell^{-2\left(\widehat{\alpha}_N-\alpha_0\right)} \ell^{-2\alpha_0}(2\ell+1)\left( \widehat{C}_{\ell;N}(0)-{C}_{\ell}(0)\right)\\&+\sum_\ell \left(\ell^{-2\left(\widehat{\alpha}_N-\alpha_0\right)}-1 \right)\ell^{-2\alpha_0}(2\ell+1) {C}_{\ell}(0) \\
		&=D_1 + D_2.
	\end{align*}
	Thus, it holds that
	\begin{equation*}
		\Pr\left(\left\vert \widehat{D}_{N} (\widehat{\alpha}_N) - D_{N} (\alpha_0) \right \vert \ge 2\varepsilon\right) \leq \Pr(|D_1| \ge \varepsilon) +  \Pr(|D_2| \ge \varepsilon).
	\end{equation*}
	For the first term, we choose a constant $\delta > 4$, and we obtain
	\begin{align*}
		\Pr(|D_1| \ge \varepsilon)  &\le  \Pr\left (|D_1| \ge \varepsilon \,  \cap \, \left \vert \widehat{\alpha}_{N}-\alpha_0\right\vert < \frac{1}{\delta} \right) + \Pr\left (\left \vert \widehat{\alpha}_{N}-\alpha_0\right\vert\ge \frac{1}{\delta} \right)\\
		&\le \Pr \left (  \sum_\ell \ell^{2/\delta} \ell^{-2\alpha_0}(2\ell+1) C_\ell(0) \left | \frac{\widehat{C}_{\ell;N}(0)}{C_\ell(0)}-1\right| \ge   \varepsilon \right)+ o(1) \\
		& \le \frac{c}{\epsilon}  \sum_\ell \ell^{2/\delta} \ell^{-2\alpha_0}(2\ell+1) C_\ell(0) \frac{1}{\sqrt{(2\ell+1)N}} + o(1) = o(1),
	\end{align*}
where $c>0$. \\On the other hand, for a suitably small $%
\delta >0$,%
\begin{eqnarray*}
\Pr \left( \left\vert D_2\right\vert \geq \varepsilon \right)
&=&\Pr \left( \left[ \left\vert D_2\right\vert \geq \varepsilon%
\right] \cap \left[ \log L_N\left| \alpha _{0}-\widehat{\alpha }_{N}\right| %
\right] <\delta \right) +\Pr \left( \log L_N\left| \alpha _{0}-\widehat{\alpha 
}_{N}\right| \geq \delta \right) \\
&=&\Pr \left( \left[ \left\vert D_2\right\vert \geq \varepsilon%
\right] \cap \left[ \log L_N\left| \alpha _{0}-\widehat{\alpha }_{N}\right| %
\right] <\delta \right) +o(1),
\end{eqnarray*}%
and using $\left\vert e^{-x}-1\right\vert \leq 2|x|$ for $|x|\le1,$ we
obtain%
\begin{equation*}
\left\vert \ell^{-2\left( \alpha _{0}-\widehat{\alpha }_{N}\right)
}-1\right\vert =\left\vert \exp (-2\log \ell\left( \alpha _{0}-\widehat{\alpha }%
_{N})\right) -1\right\vert \leq 4\log \ell\left\vert \alpha _{0}-\widehat{\alpha }%
_{N}\right\vert ;
\end{equation*}%
hence,
\begin{align*}
&\Pr \left( \left[ \left\vert D_2\right\vert \geq \frac{\varepsilon }{2}%
\right] \cap \left[ \log L_N \left| \alpha _{0}-\widehat{\alpha }_{N}\right |%
\right] <\delta \right)\\
&\leq \Pr \left( \sum_\ell \left | \ell^{-2\left(\widehat{\alpha}-\alpha_0\right)}-1 \right|\ell^{-2\alpha_0}(2\ell+1) {C}_{\ell}(0)  \geq \varepsilon\cap \left[ \log L_N\left|
\alpha _{0}-\widehat{\alpha }_{N}\right| \right] <\delta \right) \\
&\leq \Pr \left(4  \log L_N \left\vert \alpha _{0}-\widehat{\alpha }%
_{N} \right\vert  \sum_\ell   \ell^{-2\alpha_0}(2\ell+1) {C}_{\ell}(0)  \geq \varepsilon  \right) = o(1).
\end{align*}%
The term $\frac{\widehat{U}_{N} (\widehat{\alpha}_N) - U_N (\alpha_0)}{U_N (\alpha_0)}$ follows a similar argument. Then, by applying Slutsky theorem, we obtain the result.
\end{proof}

\subsection{Asymptotic Normality}
In this section, our aim is to establish the asymptotic Gaussianity for the estimator $\widehat{\alpha}_N$, following the lines driven by \cite{Robinson}, see also 
\cite{am,hayashi,NmcF}.

Recall Equation \eqref{eq:gstar}. For each $N>1$ there exists $\overline{\alpha}_N: |\overline{\alpha}_N - \alpha_0| \le |\widehat{\alpha}_N -\alpha_0|$ such
that, with probability one,
\begin{equation*}
(\widehat{\alpha}_N - \alpha_0) = - \frac{S_N(\alpha_0)}{Q_N(\overline{\alpha}_N)},
\end{equation*}
where, for a generic $\alpha \in A$,
\begin{align*}
	&	S_{N}(\alpha)= \frac{\diff}{\diff \alpha} R_{N}(G^*, \alpha), \qquad  Q_{N}(\alpha)= -\frac{\diff^2}{\diff \alpha^2} R_{N}(G^*, \alpha)=-\frac{\diff}{\diff \alpha} S_{N}(\alpha),
\end{align*}
are the score and the information function respectively.
For the sake of brevity, when it does not cause confusion, we will omit the dependence on $\alpha$. Thus, the score and the information functions are given respectively by
\begin{align*}
	&	S_{N}(\alpha) = \frac{-2\widehat{U}^\prime_{N}\widehat{U}_{N}\widehat{D}_{N}+\widehat{D}^{\prime}_{N}\widehat{U}^2_{N}}{\widehat{D}^2_{N}};\\
	& Q_{N}(\alpha)= \frac{2\widehat{U}^{\prime\prime}_{N}\widehat{U}_{N}\widehat{D}^2_{N}+2\left(\widehat{U}^\prime_N\right)^2\widehat{D}_{N}^2- \widehat{D}_{N}^{\prime\prime}\widehat{U}_{N}^2\widehat{D}_{N} -4\widehat{D}^{\prime}_{N}\widehat{D}_{N}\widehat{U}^\prime_{N}\widehat{U}_{N}+2\widehat{U}^2_{N}\left(\widehat{D}_{N}^\prime\right)^2}{\widehat{D}^3_{N}}.
\end{align*}
Define also
$$
Q(\alpha) = \frac{2U'' U D^2+2\left( U' \right)^2 D^2-D'' U^2D -4D'DU'U+2U^2\left(D' \right)^2}{D^3}.
$$
Before stating the main theorem, we introduce two ancillary results regarding the convergence of $S_N\left(\alpha_0\right)$ and $Q_N\left(\overline{\alpha}_N\right).$ 
\begin{lemma}\label{lemma:gattoS}Under Conditions \ref{cond:semiparam} and \ref{cond:tronc}, as $N \rightarrow \infty$, it holds that
\begin{equation*}
	\sqrt{N} S_{N}(\alpha_0) \overset{d}{\to}  \mathcal{N}\left (0, 	G_0^2\sum_\ell  \left (2\log \ell + \frac{D'}{D} \right)^2 \ell^{-2\alpha_0}(2\ell+1) C_\ell(0)C_{\ell;Z}   \right).
\end{equation*}
\end{lemma}
\begin{lemma}\label{lemma:gattoQ}
Under Conditions \ref{cond:semiparam} and \ref{cond:tronc}, as $N \rightarrow \infty$, it holds that
$$
Q_N(\overline{\alpha}_N) \overset{p}{\to} Q(\alpha_0),
$$
where
$$
Q(\alpha_0) = \frac{G_0^2}{2} \left ( D'' (\alpha_0 ) - \frac{\left (D'(\alpha_0) \right)^2}{D(\alpha_0)} \right) .
$$
\end{lemma}
We are now ready to state the main result described in this section.
\begin{theorem}\label{th:asn} Under Conditions \ref{cond:semiparam} and \ref{cond:tronc}, as $N \rightarrow \infty$, it holds that

	\begin{equation*}
		\sqrt{N} (\widehat{\alpha}_N - \alpha_0) \overset{d}{\to}  \mathcal{N}\left (0, 	\sigma^2(\theta)   \right),
	\end{equation*}
	where
	$$
	\sigma^2(\theta) = \frac{4}{G_0^2}\frac{\sum_{\ell\in \naturals}  \left (2\log \ell + \frac{D'(\alpha_0)}{D(\alpha_0)} \right)^2 \ell^{-2\alpha_0}(2\ell+1) C_\ell(0)C_{\ell;Z}}{\left( \sum_{\ell\in\naturals}  \left (2\log \ell + \frac{D'(\alpha_0)}{D(\alpha_0)} \right)^2 \ell^{-2\alpha_0}(2\ell+1) C_\ell(0) \right)^2}.
	$$
\end{theorem}

\begin{remark}
Note that, in $\sigma^2(\theta)$,
$$
 \sum_{\ell\in\naturals}  \left (2\log \ell + \frac{D'(\alpha_0)}{D(\alpha_0)} \right)^2 \ell^{-2\alpha_0}(2\ell+1) C_\ell(0) = D'' (\alpha_0 ) - \frac{\left (D'(\alpha_0) \right)^2}{D(\alpha_0)}.
$$
However, we prefer to keep the explicit notation to compare numerator and denominator of $\sigma^2(\theta)$.
\end{remark}

\begin{proof}[Proof of Theorem \ref{th:asn}]
	Using Lemmas \ref{lemma:gattoS} and \ref{lemma:gattoQ}, and using Slutski theorem yields the claimed result.
\end{proof}

\section{Proofs of the auxiliary results}\label{sec:salmazzi}
This section collects the proofs of the auxiliary results stated in Section \ref{sec:main}.
\begin{proof}[Proof of Lemma \ref{lemma:V}]
	Without loss of generality, take $\epsilon >0$ such that $\overline{B}_\epsilon \cap A$ is non-empty. Proving \eqref{eq:liminf} is equivalent to proving
	$$
	\lim_{N \to \infty} \inf_{\overline{B}_\epsilon \cap A} \frac{D_N(\alpha)}{D_N(\alpha_0)} \frac{U_N^2(\alpha_0)}{U_N^2(\alpha)} =  \delta_\epsilon,
	$$
	for some constant $\delta_\epsilon >1$.\\
	First, note that, for all $N>1$,
	\begin{align*}
		\frac{D_N(\alpha)}{D_N(\alpha_0)} \frac{U_N^2(\alpha_0)}{U_N^2(\alpha)}  &= \frac{\sum_\ell (2\ell+1) C_\ell(0) \ell^{-2\alpha}}{\sum_\ell (2\ell+1) C_\ell(0) \ell^{-2\alpha_0}} \left(  \frac{G_0\sum_\ell (2\ell+1) C_\ell(0) \ell^{-2\alpha_0}}{G_0\sum_\ell (2\ell+1) C_\ell(0) \ell^{-\alpha-\alpha_0}}  \right)^2\\
		&\\
		&= \frac{\sum_\ell (2\ell+1) C_\ell(0) \ell^{-2\alpha} \sum_\ell (2\ell+1) C_\ell(0) \ell^{-2\alpha_0}}{\left (\sum_\ell (2\ell+1) C_\ell(0) \ell^{-\alpha-\alpha_0}  \right)^2} \ge 1,
	\end{align*}
	by Cauchy-Schwartz inequality. In particular, equality holds if and only if $\alpha= \alpha_0$.
	Morever, for $\alpha > \alpha_0$, this quantity is monotone nondecreasing and, for $\alpha < \alpha_0$ is monotone nonincreasing. Indeed,
	\begin{align*}
		& \frac{d}{d\alpha} \left [ \frac{\sum_\ell (2\ell+1) C_\ell(0) \ell^{-2\alpha} \sum_\ell (2\ell+1) C_\ell(0) \ell^{-2\alpha_0}}{\left (\sum_\ell (2\ell+1) C_\ell(0) \ell^{-\alpha-\alpha_0}  \right)^2}  \right] 
		\\ & \quad \quad \quad \quad =\frac{\sum_\ell (2\ell+1) C_\ell(0) \ell^{-2\alpha_0}}{\left (\sum_\ell (2\ell+1) C_\ell(0) \ell^{-\alpha-\alpha_0}\right)^4} \Bigg [  -2\sum_\ell (2\ell+1) C_\ell(0) \ell^{-2\alpha} \log \ell \left( \sum_\ell (2\ell+1) C_\ell(0) \ell^{-\alpha-\alpha_0} \right)^2 \\&\quad \quad \quad \quad+ 2\sum_\ell (2\ell+1) C_\ell(0) \ell^{-2\alpha} \sum_\ell (2\ell+1) C_\ell(0) \ell^{-\alpha-\alpha_0}   \sum_\ell (2\ell+1) C_\ell(0) \ell^{-\alpha-\alpha_0}   \log \ell  \Bigg],
	\end{align*}
	and this is nonnegative if and only if
	\begin{align*}
		-\sum_\ell (2\ell+1) C_\ell(0) \ell^{-2\alpha} \log \ell  \sum_\ell (2\ell+1) C_\ell(0) \ell^{-\alpha-\alpha_0}  + \sum_\ell (2\ell+1) C_\ell(0) \ell^{-2\alpha}  \sum_\ell (2\ell+1) C_\ell(0) \ell^{-\alpha-\alpha_0}   \log \ell  \ge 0,
	\end{align*}
	that is
	\begin{align*}
		\frac{\sum_\ell (2\ell+1) C_\ell(0) \ell^{-\alpha-\alpha_0} \log \ell }{\sum_\ell (2\ell+1) C_\ell(0) \ell^{-\alpha-\alpha_0}}  \ge \frac{\sum_\ell (2\ell+1) C_\ell(0) \ell^{-2\alpha} \log \ell }{\sum_\ell (2\ell+1) C_\ell(0) \ell^{-2\alpha}} .
	\end{align*}
	Thus, we have two weighted sums with weights respectively $w_\ell = \frac{ (2\ell+1) C_\ell(0) \ell^{-\alpha-\alpha_0} }{\sum_\ell (2\ell+1) C_\ell(0) \ell^{-\alpha-\alpha_0}} $ and $w'_\ell =  \frac{(2\ell+1) C_\ell(0) \ell^{-2\alpha} }{\sum_\ell (2\ell+1) C_\ell(0) \ell^{-2\alpha}} $, $\sum_\ell w_\ell = \sum_\ell w'_\ell = 1$. 
	
	If $\alpha > \alpha_0$, for all $\ell_1 \le \ell_2$, it holds that $w'_{\ell_1}/ w'_{\ell_2} \ge w_{\ell_1}/ w_{\ell_2}$, which implies, by \cite[Theorem 2]{weightedsum},
	$$
	\sum_\ell w_\ell \log \ell \ge \sum_\ell w'_\ell \log \ell.
	$$
	The case $\alpha < \alpha_0$ follows similar arguments.
Hence, since $\overline{B}_\epsilon \cap A$ is compact, there exists $b_{\epsilon,1}, b_{\epsilon,2}\in \overline{B}_\epsilon \cap A$, such that
	$$
	\inf_{\overline{B}_\epsilon \cap A} \frac{D_N(\alpha)}{D_N(\alpha_0)} \frac{U^2_N(\alpha_0)}{U^2_N(\alpha)} =  \min\left  \{ \frac{D_N(b_{\epsilon,1})}{D_N(\alpha_0)} \frac{U^2_N(\alpha_0)}{U^2_N(b_{\epsilon,1})} , \frac{D_N(b_{\epsilon,2})}{D_N(\alpha_0)} \frac{U^2_N(\alpha_0)}{U^2_N(b_{\epsilon,2})}  \right\}.
	$$
	Moreover, since it must be $b_{\epsilon,1}, b_{\epsilon,2} \ne \alpha_0$,
		\begin{align*}
		\lim_{N\to \infty } \inf_{\overline{B}_\epsilon \cap A} \frac{D_N(\alpha)}{D_N(\alpha_0)} \frac{U^2_N(\alpha_0)}{U^2_N(\alpha)}  &= \lim_{N\to \infty } \min\left  \{ \frac{D_N(b_{\epsilon,1})}{D_N(\alpha_0)} \frac{U^2_N(\alpha_0)}{U^2_N(b_{\epsilon,1})} , \frac{D_N(b_{\epsilon,2})}{D_N(\alpha_0)} \frac{U^2_N(\alpha_0)}{U^2_N(b_{\epsilon,2})}  \right\} \\&=  \min\left  \{  \lim_{N\to \infty } \frac{D_N(b_{\epsilon,1})}{D_N(\alpha_0)} \frac{U^2_N(\alpha_0)}{U^2_N(b_{\epsilon,1})} ,  \lim_{N\to \infty } \frac{D_N(b_{\epsilon,2})}{D_N(\alpha_0)} \frac{U^2_N(\alpha_0)}{U^2_N(b_{\epsilon,2})}  \right\} \\&= \delta_\epsilon >1,
	\end{align*}
		which gives the claimed result.
\end{proof}

\begin{proof}[Proof of Lemma \ref{lemma:T}]
	Note that, since $A = [a_1, a_2]$,
	\begin{align*}
		\sup_A \left | \frac{\widehat{U}_{N} (\alpha) - U_{N}(\alpha)}{U_{N}(\alpha)} \right |&= \sup_A \left |\frac{\sum_\ell \ell^{-\alpha} ( 2\ell+1) C_\ell(1) \left[ \frac{\widehat{C}_{\ell;N}(1)}{C_\ell(1)} - 1 \right]}{\sum_\ell \ell^{-\alpha} (2\ell+1) C_{\ell}(1)}  \right | \\
		&\le \frac{\sup_A \left | \sum_\ell \ell^{-\alpha} ( 2\ell+1) C_\ell(1) \left[ \frac{\widehat{C}_{\ell;N}(1)}{C_\ell(1)} - 1 \right] \right |}{\inf_A \left | \sum_\ell \ell^{-\alpha} (2\ell+1) C_{\ell}(1)  \right |} \\
		&\le \frac{ \sum_\ell \ell^{-a_1-\alpha_0} ( 2\ell+1) C_\ell(0) \left | \frac{\widehat{C}_{\ell;N}(1)}{C_\ell(1)} - 1 \right| }{ \sum_\ell \ell^{-a_2-\alpha_0} (2\ell+1) C_{\ell}(0) } .
	\end{align*}
	Moreover, from \cite[Lemma 2 (Supplementary material)]{cm19},
	$$
	\Ex \left | \frac{\widehat{C}_{\ell;N}(1)}{C_\ell(1)} - 1 \right| \le \frac{c}{\sqrt{(2\ell+1)N}},
	$$
	where $c>0$. 
	Hence, we have that 
	$$
	\Ex \left [ \sup_A \left | \frac{\widehat{U}_{N} (\alpha) - U_{N}(\alpha)}{U_{N}(\alpha)} \right |   \right] = O \left( \frac{1}{\sqrt{N}} \right).
	$$
	The proof for the $ \Ex \left [\sup_A \left | \frac{\widehat{D}_{N} (\alpha) - D_{N}(\alpha)}{D_{N}(\alpha)} \right | \right]$ follows the same lines. 
\end{proof}
\begin{proof}[Proof of Lemma \ref{Ghat0}]
	Consider the quantity
	\begin{align*}
		\Ex \left | \frac{\sum_\ell \ell^{-\alpha_0} ( 2\ell+1) \widehat{C}_{\ell;N}(1)}{\sum_\ell \ell^{-2\alpha_0} (2\ell+1) C_{\ell}(0)} - G_0 \right |&= \Ex \left |\frac{\sum_\ell \ell^{-\alpha_0} ( 2\ell+1) C_\ell(0) \left[ \frac{\widehat{C}_{\ell;N}(1)}{C_\ell(0)} - G_0 \ell^{-\alpha_0}\right]}{\sum_\ell \ell^{-2\alpha_0} (2\ell+1) C_{\ell}(0)}  \right | \\
		&\le \frac{\sum_\ell \ell^{-\alpha_0} ( 2\ell+1) C_\ell(0)\Ex\left| \frac{\widehat{C}_{\ell;N}(1)}{C_\ell(0)} - G_0 \ell^{-\alpha_0}\right|}{\sum_\ell \ell^{-2\alpha_0} (2\ell+1) C_{\ell}(0)} \\
		& \le \frac{1}{\sqrt{N}} \frac{\sum_\ell \ell^{-\alpha_0} \sqrt{ 2\ell+1} C_\ell(0)}{\sum_\ell \ell^{-2\alpha_0} (2\ell+1) C_{\ell}(0)} \\
		&\le \frac{c^\prime}{\sqrt{N}},
	\end{align*}
	where $c^\prime>0$, from \cite[Lemma 2 (Supplementary material)]{cm19}. Moreover,
	\begin{align*}
		\Ex \left |\sum_\ell \ell^{-2\alpha_0} (2\ell+1) \widehat{C}_{\ell;N}(0) - \sum_\ell \ell^{-2\alpha_0} (2\ell+1) C_{\ell}(0) \right | \le \sum_\ell \ell^{-2\alpha_0} (2\ell+1)C_\ell(0)  \Ex \left| \frac{\widehat{C}_{\ell;N}(0)}{C_\ell(0)} -1 \right |
		& \le \frac{c^{\prime\prime}}{\sqrt{N}},
	\end{align*}
	where $c^{\prime\prime}>0$, from \cite[Lemma 1 (Supplementary material)]{cm19}.
\end{proof}

\begin{proof}[Proof of Lemma \ref{lemma:loglog}]
	For $\epsilon > 0$, set $B_\epsilon=\{\alpha: |\alpha -\alpha_0| < \epsilon\}$ and $M_{\epsilon}=\{\alpha: \log L_N |\alpha -\alpha_0| < \epsilon\}$, and observe that $\overline{B}_\epsilon \subset \overline{M}_\epsilon$. Then,
	\begin{align*}
		\Pr \left( \log L_N | \widehat{\alpha}_N - \alpha _{0} | \geq \epsilon \right) &= \Pr \left( \widehat{\alpha}_N \in \overline{M}_\epsilon \cap A \right)\le \Pr \left( \widehat{\alpha}_N \in B_\epsilon \cap \overline{M}_\epsilon \cap A \right) + \Pr \left( \widehat{\alpha}_N \in \overline{B}_\epsilon \cap A \right).
	\end{align*}
	For the first term on the right hand side, we have
	\begin{align*}
		\Pr \left( \widehat{\alpha}_N \in B_\epsilon \cap \overline{M}_\epsilon \cap A \right) &\le \Pr  \left ( \sup_{B_\epsilon \cap \overline{M}_\epsilon \cap A } \widetilde{R}_{N} (\alpha) \ge  \widetilde{R}_{N} (\alpha_0) \right) \\
		&=   \Pr  \left ( \sup_{B_\epsilon \cap \overline{M}_\epsilon \cap A } \Delta \widetilde R_{N} (\alpha, \alpha_0) \ge  0 \right)\\
		&\le \Pr  \left ( \sup_{ A}  |T_{N} (\alpha, \alpha_0)| \ge  \inf_{B_\epsilon \cap \overline{M}_\epsilon \cap A } V_{N}(\alpha, \alpha_0)   \right).
	\end{align*}
	Observe also that $B_\epsilon \cap \overline{M}_\epsilon=\{\alpha: \epsilon/ \log L_N \le |\alpha -\alpha_0| < \epsilon \}$. We have proved in Lemma \ref{lemma:V} that the quantity $V_{N}(\alpha, \alpha_0)$ is monotone nondecreasing for $\alpha > \alpha_0$ and monotone nonincreasing for $\alpha < \alpha_0$. Hence, for $N$ sufficiently large, the infimum over $B_\epsilon \cap \overline{M}_\epsilon \cap A$ is reached at $m_{\epsilon,1}=\alpha_0 - \frac{\epsilon}{\log L_N}$ or at $m_{\epsilon,2}=\alpha_0 + \frac{\epsilon}{\log L_N}$. 	For the sake simplicity, we will use $\epsilon_N$ to indicate $\frac{\epsilon}{\log L_N}$.\\
	Let us first consider 
	$$
	V_{N}(m_{\epsilon,1}, \alpha_0) = \log \left (  \frac{\sum_\ell (2\ell+1) C_\ell(0) \ell^{-2\alpha_0 + 2\epsilon_N } \sum_\ell (2\ell+1) C_\ell(0) \ell^{-2\alpha_0}}{\left (\sum_\ell (2\ell+1) C_\ell(0) \ell^{-2\alpha_0 + \epsilon_N}  \right)^2} -1 +1 \right),
	$$
	and note that
	\begin{align*}
		&\frac{\sum_\ell (2\ell+1) C_\ell(0) \ell^{-2\alpha_0 + 2\epsilon_N } \sum_\ell (2\ell+1) C_\ell(0) \ell^{-2\alpha_0}}{\left (\sum_\ell (2\ell+1) C_\ell(0) \ell^{-2\alpha_0 + \epsilon_N}  \right)^2} -1  \\=&  \frac{\sum_\ell (2\ell+1) C_\ell(0) \ell^{-2\alpha_0 + 2\epsilon_N } \sum_\ell (2\ell+1) C_\ell(0) \ell^{-2\alpha_0} - \left (\sum_\ell (2\ell+1) C_\ell(0) \ell^{-2\alpha_0 + \epsilon_N}  \right)^2}{\left (\sum_\ell (2\ell+1) C_\ell(0) \ell^{-2\alpha_0 + \epsilon_N}  \right)^2} .
	\end{align*}
	For generic vectors $x,y \in \mathbb{R}^d$, $d \in \mathbb{N}$, with standard inner product and norm, we have the following equality
	\begin{align*}
		\frac{4 \|x\|^2 \|y\|^2 - 4\langle x, y \rangle^2}{4\langle x, y \rangle^2} = \frac{\|x-y\|^2\|x+y\|^2 - (\|x\|^2 - \|y\|^2)^2}{4\langle x, y \rangle^2}. 
	\end{align*}
	Hence, if
	$$
	x = (\sqrt{(2\ell+1)C_{\ell}(0)} \ell^{-\alpha_0 + \epsilon_N}, \ \ell = 1,\dots,L_N)^{\texttt T}, \qquad y = (\sqrt{(2\ell+1)C_{\ell}(0)}\ell^{-\alpha_0}, \ \ell = 1,\dots,L_N)^{\texttt T},
	$$
	we have
	\begin{align*}
		& 4\sum_\ell (2\ell+1) C_\ell(0) \ell^{-2\alpha_0 + 2\epsilon_N } \sum_\ell (2\ell+1) C_\ell(0) \ell^{-2\alpha_0} - 4\left (\sum_\ell (2\ell+1) C_\ell(0) \ell^{-2\alpha_0 + \epsilon_N}  \right)^2\\
		=& \sum_\ell (2\ell+1) C_\ell(0) \ell^{-2\alpha_0 } \left (\ell^{\epsilon_N}  - 1 \right)^2  \sum_\ell (2\ell+1) C_\ell(0) \ell^{-2\alpha_0} \left (\ell^{\epsilon_N}  + 1\right )^2  - \left (\sum_\ell (2\ell+1) C_\ell(0) \ell^{-2\alpha_0 }\left  (\ell^{2\epsilon_N}  - 1\right )   \right)^2.
	\end{align*}
	For the sake of simplicity, write $a_{\ell} = (2\ell+1)C_\ell(0) \ell^{-2 \alpha_0}$, then
	\begin{align*}
		&\sum_\ell (2\ell+1) C_\ell(0) \ell^{-2\alpha_0 } \left (\ell^{\epsilon_N}  - 1 \right)^2  \sum_\ell (2\ell+1) C_\ell(0) \ell^{-2\alpha_0} \left (\ell^{\epsilon_N}  + 1\right )^2  - \left (\sum_\ell (2\ell+1) C_\ell(0) \ell^{-2\alpha_0 }\left  (\ell^{2\epsilon_N}  - 1\right )   \right)^2\\
		=& \sum_{\ell_1,\ell_2} a_{\ell_1}a_{\ell_2} \left [\left (\ell_1^{\epsilon_N}  - 1 \right)^2 \left (\ell_2^{\epsilon_N}  + 1 \right)^2 - \left (\ell_1^{2\epsilon_N}  - 1 \right) \left (\ell_2^{2\epsilon_N}  - 1 \right)  \right]\\
		=& 2\sum_{\ell_1,\ell_2}  a_{\ell_1}a_{\ell_2} \ell_1^{2\epsilon_N}  + 2\sum_{\ell_1,\ell_2}  a_{\ell_1}a_{\ell_2} \ell_2^{2\epsilon_N} -2\sum_{\ell_1,\ell_2}  a_{\ell_1}a_{\ell_2} \ell_1^{\epsilon_N}  + 2\sum_{\ell_1,\ell_2}  a_{\ell_1}a_{\ell_2} \ell_2^{\epsilon_N} \\
		-&2\sum_{\ell_1,\ell_2}  a_{\ell_1}a_{\ell_2}\ell_1^{\epsilon_N}\ell_2^{2\epsilon_N}   + 2\sum_{\ell_1,\ell_2}  a_{\ell_1}a_{\ell_2} \ell_1^{2\epsilon_N}\ell_2^{\epsilon_N} - 4\sum_{\ell_1,\ell_2}  a_{\ell_1}a_{\ell_2} \ell_1^{\epsilon_N}\ell_2^{\epsilon_N}\\
		=& 2\sum_{\ell_1,\ell_2}  a_{\ell_1}a_{\ell_2}  (\ell_1^{\epsilon_N}  - \ell_2^{\epsilon_N} )^2,
	\end{align*}
	since
	\begin{align*}
		\left (\ell_1^{\epsilon_N}  - 1 \right)^2 \left (\ell_2^{\epsilon_N}  + 1 \right)^2 - \left (\ell_1^{2\epsilon_N}  - 1 \right) \left (\ell_2^{2\epsilon_N}  - 1 \right)  &= 2\ell_1^{2\epsilon_N} + 2 \ell_2^{2\epsilon_N} - 2 \ell_1^{\epsilon_N} + 2 \ell_2^{\epsilon_N}\\
		&-2\ell_1^{\epsilon_N}\ell_2^{2\epsilon_N} + 2 \ell_1^{2\epsilon_N}\ell_2^{\epsilon_N} - 4 \ell_1^{\epsilon_N}\ell_2^{\epsilon_N} .
	\end{align*}
	Without loss of generality, consider $\ell_1 < \ell_2$, thus
	$$
	(\ell_1^{\epsilon_N}  - \ell_2^{\epsilon_N} )^2 = \ell_1^{2\epsilon_N} \left (\left (\frac{\ell_2}{\ell_1}\right )^{\epsilon_N}  - 1 \right)^2 \ge \epsilon_N^2 (\log \ell_2 - \log \ell_1)^2,
	$$
	where we used the inequality $e^x - 1\ge x$, $x \ge 0$.
	This implies
	$$
	\frac{\sum_\ell (2\ell+1) C_\ell(0) \ell^{-2\alpha_0 + 2\epsilon_N } \sum_\ell (2\ell+1) C_\ell(0) \ell^{-2\alpha_0}}{\left (\sum_\ell (2\ell+1) C_\ell(0) \ell^{-2\alpha_0 + \epsilon_N}  \right)^2} \ge 1 + \frac{c_\epsilon }{ (\log L_N)^2},
	$$
	with $c_\epsilon > 0$. Now, since $$
	\log( 1 + x ) \ge \frac{x}{1+x}, \qquad x > -1,
	$$
	we have
	$$
	V_{N}(m_{\epsilon,1}, \alpha_0) \ge \frac{c_\epsilon}{1+ c_\epsilon} (\log L_N)^{-2}.
	$$
	A similar argument holds for $m_{\epsilon,2}.$
Then, from Equation \eqref{eq:trapizzino} and Lemma \ref{lemma:T},
$$
Pr  \left ( \sup_{ A}  |T_{N} (\alpha, \alpha_0)| \ge  \inf_{B_\epsilon \cap \overline{M}_\epsilon \cap A } V_{N}(\alpha, \alpha_0)   \right) = O \left ( \frac{(\log L_N)^2 }{\sqrt{N}} \right).
$$
Using Condition \ref{cond:tronc}, we obtain the claimed result. 
\end{proof}

\begin{proof}[Proof of Lemma \ref{lemma:gattoS} ]
	First, observe that
	\begin{align*}
		\frac{d R_{N} (G^* , \alpha )}{d \alpha}
		&= \frac{d}{d \alpha } \left ( - \frac{\widehat{U}_N^2}{\widehat{D}_N} \right) \\
		&= \frac{-2 \widehat{U}_N \widehat{U}_N' \widehat{D}_N + \widehat{U}_N^2 \widehat{D}_N'}{\widehat{D}_N^2} \\
		&= \frac{\widehat{U}_N}{\widehat{D}_N} \left ( -2 \widehat{U}_N' + \frac{\widehat{U}_N}{\widehat{D}_N} \widehat{D}_N' \right ).
	\end{align*}
Then, recall from Lemma \ref{Ghat0} that $\frac{\widehat{U}_N(\alpha_0)}{\widehat{D}_N(\alpha_0)} \to G_0$.
Now, we have that
	\begin{equation*}
		-2\widehat{U}_N'(\alpha_0) = 2  \sum_\ell \log \ell \, \ell^{-\alpha_0} (2\ell+1)\widehat{C}_{\ell;N}(1), 
	\end{equation*}
while
	\begin{equation*}
		\widehat{D}_N'(\alpha_0) = -2  \sum_\ell \log \ell \, \ell^{-2\alpha_0}(2\ell+1) \widehat{C}_{\ell;N}(0). 
	\end{equation*}
	The following upper bound thus holds
	\begin{align*}
		&\Ex \left |\sum_\ell \log \ell \, \ell^{-2\alpha_0} (2\ell+1) \widehat{C}_{\ell;N}(0) - \sum_\ell  \log \ell \,  \ell^{-2\alpha_0} (2\ell+1) C_{\ell}(0) \right | \\
		& \quad\quad\quad\quad \quad\le \sum_\ell \log \ell \,  \ell^{-2\alpha_0} (2\ell+1)C_\ell(0)  \Ex \left| \frac{\widehat{C}_{\ell;N}(0)}{C_\ell(0)} -1 \right |
		 \le \frac{c}{\sqrt{N}},
	\end{align*}
where $c>0$.
	%
	%
	%
	In addition,
	\begin{align*}
2 \sum_\ell \log \ell \, \ell^{-\alpha_0} (2\ell+1) \widehat{C}_{\ell;N}(1) +G_0\widehat{D}_N' &= 2 \left (\sum_\ell \log \ell \, \ell^{-\alpha_0} (2\ell+1) \widehat{C}_{\ell;N}(1)  - \sum_\ell \log \ell \, \ell^{-2\alpha_0}(2\ell+1) \widehat{C}_{\ell;N}(0) G_0 \right)\\
		&= 2 \sum_\ell \log \ell \, \ell^{-\alpha_0} (2\ell+1)
		\left(\widehat{C}_{\ell;N}(1) - G_0 \ell^{-\alpha_0 } \widehat{C}_{\ell;N}(0) \right)\\
		&= \frac{2}{\sqrt{N}} \sum_\ell \log \ell \, \ell^{-\alpha_0} \sqrt{2\ell+1}C_\ell(0)B_{\ell;N},
	\end{align*}
	where
$$
B_{\ell,N} = \frac{1}{C_{\ell}(0)\sqrt{N(2\ell+1)}} {\sum_{t=1}^N \sum_{m=-\ell}^\ell a_{\ell,m} (t-1) a_{\ell,m;Z} (t)}.
$$
	Recall that, for fixed $\ell \in \naturals$, $B_{\ell;N}$ belongs to the second order Wiener chaos and
	$$
	B_{\ell;N}  \overset{d}{\to} Z_\ell \sim  \mathcal{N}\left(0, \frac{C_{\ell;Z}}{C_\ell(0)} \right), \qquad N \to \infty;
	$$
	see \cite{noupebook} and also \cite{cm19}, for definitions and proofs.
	Moreover, we have that 	
	\begin{align*}
		\frac{\widehat{U}_N}{\widehat{D}_N} \widehat{D}_N' - G_0\widehat{D}_N' &=\frac{\widehat{D}_N'}{\widehat{D}_N}  \left (\sum_\ell  \ell^{-\alpha_0} (2\ell+1) \widehat{C}_{\ell;N}(1)  - \sum_\ell  \ell^{-2\alpha_0}(2\ell+1) \widehat{C}_{\ell;N}(0) G_0 \right)\\
		&= \frac{\widehat{D}_N'}{\widehat{D}_N}  \sum_\ell  \ell^{-\alpha_0} (2\ell+1)
		\left(\widehat{C}_{\ell;N}(1) - G_0 \ell^{-\alpha_0 } \widehat{C}_{\ell;N}(0) \right)\\
		&= \frac{1}{\sqrt{N}} \frac{\widehat{D}_N'}{\widehat{D}_N} \sum_\ell  \ell^{-\alpha_0} \sqrt{2\ell+1}C_\ell(0)B_{\ell;N}.
	\end{align*}
	Hence, we can define
	\begin{equation*}
A_N=		\sqrt{N} \left ( -2 \widehat{U}_N' + \frac{\widehat{U}_N}{\widehat{D}_N} \widehat{D}_N' \right ) = \sum_\ell  \ell^{-\alpha_0} \sqrt{2\ell+1}C_\ell(0)B_{\ell;N} \left (2\log \ell + \frac{\widehat{D}_N'}{\widehat{D}_N} \right)
	\end{equation*}
	while
	\begin{equation*}
		C_{N} = \sum_\ell  \ell^{-\alpha_0} \sqrt{2\ell+1}C_\ell(0)B_{\ell;N} \left (2\log \ell + \frac{D'}{D} \right),
	\end{equation*}
	where
	\begin{equation*}
		\frac{D'(\alpha_0)}{D(\alpha_0)} = - 2 \frac{ \sum_\ell \log \ell \, \ell^{-2\alpha_0}(2\ell+1) C_\ell(0) }{\sum_\ell  \ell^{-2\alpha_0}(2\ell+1) C_\ell(0)} < \infty.
	\end{equation*}
Thus, we first prove that $A_{N} - C_{N} \to 0$ in probability, as $N \to \infty$. Consider
	\begin{equation*}
		A_{N} - C_{N} = \left (\frac{\widehat{D}_N'}{\widehat{D}_N} -  \frac{D'}{D} \right)  \sum_\ell  \ell^{-\alpha_0} \sqrt{2\ell+1}C_\ell(0)B_{\ell;N}.
	\end{equation*}
	It is clear that
	\begin{align*}
		\left (\frac{\widehat{D}_N'}{\widehat{D}_N} -  \frac{D'}{D} \right) \overset{p}{\to} 0;
	\end{align*}
	moreover,
	\begin{align}\label{eq:AN-CN}
		\sum_\ell  \ell^{-\alpha_0} \sqrt{2\ell+1}C_\ell(0)B_{\ell;N} \overset{d}{\to} \mathcal{N}\left (0, 	\sum_\ell \ell^{-2\alpha_0}(2\ell+1) C_\ell(0)C_{\ell;Z}   \right).
	\end{align}
Indeed, from \cite{cm19}, we have that	
	\begin{align*}
	&\mathbb{E} \left[ \sum_\ell  \ell^{-\alpha_0} \sqrt{2\ell+1}C_\ell(0)B_{\ell;N} \right] = 0,\\
		&\operatorname{Var} \left [ \sum_\ell  \ell^{-\alpha_0} \sqrt{2\ell+1}C_\ell(0)B_{\ell;N} \right]  =  \sum_\ell  \ell^{-2\alpha_0} (2\ell+1) C_\ell(0)C_{\ell;Z},
	\end{align*}
	which is convergent, and
	\begin{align*}
		\operatorname{Cum}_4 \left [ \sum_\ell  \ell^{-\alpha_0} \sqrt{2\ell+1}C_\ell(0)B_{\ell;N} \right]  &=  \sum_\ell  \ell^{-4\alpha_0} (2\ell+1)^2 C^4_\ell(0)\operatorname{Cum}_4[B_{\ell;N} ]\\
		&= \frac{6}{N}  \sum_\ell  \ell^{-4\alpha_0} (2\ell+1)^2 C^2_\ell(0) C^2_{\ell;Z},
	\end{align*}
	which goes to 0 as $N \to \infty$. Thus, by the Fourth Moment Theorem \cite[Theorem 5.2.7]{noupebook}, we obtain \eqref{eq:AN-CN}.
Finally, as far as the convergence in distribution for $C_{N}$ is concerned, it holds that
	\begin{equation*}
		C_{N} \overset{d}{\to} \mathcal{N}\left (0, 	\sum_\ell  \left (2\log \ell + \frac{D'}{D} \right)^2 \ell^{-2\alpha_0}(2\ell+1) C_\ell(0)C_{\ell;Z}   \right),
	\end{equation*}
where the proof is similar to the previous one. \\	
As a consequence of Slutsky theorem, we conclude that
	\begin{equation*}
		\sqrt{N} S_{N}(\alpha_0) \overset{d}{\to}  \mathcal{N}\left (0, 	G_0^2\sum_\ell  \left (2\log \ell + \frac{D'}{D} \right)^2 \ell^{-2\alpha_0}(2\ell+1) C_\ell(0)C_{\ell;Z}   \right).
	\end{equation*}
\end{proof}

\begin{proof}[Proof of Lemma \ref{lemma:gattoQ}]
	We first need to prove that $Q(\alpha_0) \ne 0$. 
	If we denote with $a_\ell = (2\ell+1)C_\ell(0)\ell^{-2\alpha_0}$, we can write
	\begin{align*}
		Q(\alpha_0) &= \frac{G_0^2}{D^3}   \sum_{\ell_1, \ell_2, \ell_3, \ell_4} \frac{a_{\ell_1}}{1-\phi_{\ell_1}^2}  \frac{a_{\ell_2}}{1-\phi_{\ell_2}^2}  \frac{a_{\ell_3}}{1-\phi_{\ell_3}^2}  \frac{a_{\ell_4}}{1-\phi_{\ell_4}^2}  \left (2\log ^2 \ell_1 + 2 \log \ell_1 \log \ell_2 -8 \log \ell_1 \log \ell_2 + 8 \log \ell_1 \log \ell_2 - 4 \log ^2 \ell_1 \right ) \\
		& =   \frac{G_0^2}{D^3}   \sum_{\ell_1, \ell_2, \ell_3, \ell_4} \frac{a_{\ell_1}}{1-\phi_{\ell_1}^2}  \frac{a_{\ell_2}}{1-\phi_{\ell_2}^2}  \frac{a_{\ell_3}}{1-\phi_{\ell_3}^2}  \frac{a_{\ell_4}}{1-\phi_{\ell_4}^2}  \left (- 2\log ^2 \ell_1 + 2 \log \ell_1 \log \ell_2  \right ) \\
		& = - 2 \frac{G_0^2}{D^3}   \sum_{\ell_1 >\ell_2, \ell_3,  \ell_4} \frac{a_{\ell_1}}{1-\phi_{\ell_1}^2}  \frac{a_{\ell_2}}{1-\phi_{\ell_2}^2}  \frac{a_{\ell_3}}{1-\phi_{\ell_3}^2}  \frac{a_{\ell_4}}{1-\phi_{\ell_4}^2}  \left (\log \ell_1 - \log \ell_2  \right )^2 < 0.
	\end{align*}
	A simplified expression is then given by
	$$
	Q(\alpha_0) = \frac{G_0^2}{2} \left ( D''(\alpha_0) -  \frac{\left (D'(\alpha_0) \right)^2}{D(\alpha_0)} \right).
	$$
	The rest of the proof follows the arguments of Theorem \ref{th:cons}, by proving separately the convergence of each term in $Q_N(\overline{\alpha}_N)$ and then applying Slutsky theorem.
	Note that $|\overline{\alpha}_N - \alpha_0|=o_p(1)$ and $\log L_N |\overline{\alpha}_N - \alpha_0|=o_p(1)$. Indeed, $|\overline{\alpha}_N - \alpha_0| \le |\widehat{\alpha}_N - \alpha_0|$, thus for $\epsilon > 0$
	$$
	\Pr \left( \left | \overline{\alpha}_N - \alpha _{0} \right |\geq \epsilon \right) \le \Pr \left(  \left | \widehat{\alpha}_N - \alpha _{0} \right | \geq \epsilon \right),
	$$
	and
	$$
	\Pr \left( \log L_N \left | \overline{\alpha}_N - \alpha _{0} \right | \geq \epsilon \right) \le \Pr \left( \log L_N \left |\widehat{\alpha}_N - \alpha _{0} \right |\geq \epsilon \right).
	$$
	We report here for completeness the proof for the term $D''(\overline{\alpha}_N)$.
	We have that
	\begin{align*}
	\widehat{D}''_{N} (\overline{\alpha}_N) - D''_N (\alpha_0)&=	\sum_\ell \log^2 \ell\, \ell^{-2\overline{\alpha}_N} (2\ell+1) \widehat{C}_{\ell;N}(0)-\sum_\ell \log^2 \ell \, \ell^{-2{\alpha}_0} (2\ell+1){C}_{\ell}(0)\\ &=  \sum_\ell  \log^2 \ell \, \ell^{-2(\overline{\alpha}_N-\alpha_0)} \ell^{-2\alpha_0} (2\ell+1) \widehat{C}_{\ell;N}(0)-\sum_\ell  \log^2 \ell \, \ell^{-2{\alpha}_0} (2\ell+1){C}_{\ell}(0)  \\
		&\pm  \sum_\ell  \log^2 \ell \, \ell^{-2(\overline{\alpha}_N-\alpha_0)} \ell^{-2\alpha_0} (2\ell+1) C_{\ell}(0)\\
		&= \sum_\ell   \log^2 \ell \,\ell^{-2\left(\overline{\alpha}_N-\alpha_0\right)} \ell^{-2\alpha_0}(2\ell+1)\left( \widehat{C}_{\ell;N}(0)-{C}_{\ell}(0)\right)\\&+\sum_\ell  \log^2 \ell \, \left(\ell^{-2\left(\overline{\alpha}_N-\alpha_0\right)}-1 \right)\ell^{-2\alpha_0}(2\ell+1) {C}_{\ell}(0) \\
		&=D''_1 + D''_2.
	\end{align*}
	Thus, it holds that
	\begin{equation*}
		\Pr\left(\left\vert \widehat{D}''_{N} (\overline{\alpha}_N) - D''_{N} (\alpha_0) \right \vert \ge 2\varepsilon\right) \leq \Pr(|D''_1| \ge \varepsilon) +  \Pr(|D''_2| \ge \varepsilon).
	\end{equation*}
	For the first term, we choose a constant $\delta > 4$, and we obtain
	\begin{align*}
		\Pr(|D''_1| \ge \varepsilon)  &\le  \Pr\left (|D''_1| \ge \varepsilon \,  \cap \, \left \vert \overline{\alpha}_{N}-\alpha_0\right\vert < \frac{1}{\delta} \right) + \Pr\left (\left \vert \overline{\alpha}_{N}-\alpha_0\right\vert\ge \frac{1}{\delta} \right)\\
		&\le \Pr \left (  \sum_\ell  \log^2 \ell \, \ell^{2/\delta} \ell^{-2\alpha_0}(2\ell+1) C_\ell(0) \left | \frac{\widehat{C}_{\ell;N}(0)}{C_\ell(0)}-1\right| \ge   \varepsilon \right)+ o(1) \\
		& \le \frac{c}{\epsilon}  \sum_\ell  \log^2 \ell \, \ell^{2/\delta} \ell^{-2\alpha_0}(2\ell+1) C_\ell(0) \frac{1}{\sqrt{(2\ell+1)N}} + o(1) = o(1),
	\end{align*}
where $c>0$. \\On the other hand, for a suitably small $%
\delta >0$,%
\begin{eqnarray*}
\Pr \left( \left\vert D''_2\right\vert \geq \varepsilon \right)
&=&\Pr \left( \left[ \left\vert D''_2\right\vert \geq \varepsilon%
\right] \cap \left[ \log L_N\left| \alpha _{0}-\overline{\alpha}_{N}\right| %
\right] <\delta \right) +\Pr \left( \log L_N\left| \alpha _{0}-\overline{\alpha 
}_{N}\right| \geq \delta \right) \\
&=&\Pr \left( \left[ \left\vert D''_2\right\vert \geq \varepsilon%
\right] \cap \left[ \log L_N\left| \alpha _{0}-\overline{\alpha}_{N}\right| %
\right] <\delta \right) +o(1),
\end{eqnarray*}%
and using $\left\vert e^{-x}-1\right\vert \leq 2|x|$ for $|x|\le1,$ we
obtain%
\begin{equation*}
\left\vert \ell^{-2\left( \alpha _{0}-\overline{\alpha}_{N}\right)
}-1\right\vert =\left\vert \exp (-2\log \ell\left( \alpha _{0}-\overline{\alpha}%
_{N})\right) -1\right\vert \leq 4\log \ell\left\vert \alpha _{0}-\overline{\alpha}%
_{N}\right\vert ;
\end{equation*}
hence,
\begin{align*}
&\Pr \left( \left[ \left\vert D''_2\right\vert \geq \frac{\varepsilon }{2}%
\right] \cap \left[ \log L_N \left| \alpha _{0}-\overline{\alpha}_{N}\right |%
\right] <\delta \right)\\
&\leq \Pr \left( \sum_\ell \log^2 \ell \, \left | \ell^{-2\left(\overline{\alpha}-\alpha_0\right)}-1 \right|\ell^{-2\alpha_0}(2\ell+1) {C}_{\ell}(0)  \geq \varepsilon\cap \left[ \log L_N\left|
\alpha _{0}-\overline{\alpha}_{N}\right| \right] <\delta \right) \\
&\leq \Pr \left(4  \log L_N \left\vert \alpha _{0}-\overline{\alpha}%
_{N} \right\vert  \sum_\ell  \log^2 \ell \, \ell^{-2\alpha_0}(2\ell+1) {C}_{\ell}(0)  \geq \varepsilon  \right) = o(1).
\end{align*}%
\end{proof}

\section*{Acknowledgement} 
The authors wish to thank Domenico Marinucci for many insightful discussions and suggestions.

\bibliography{Bibliography}

\end{document}